\documentclass[12pt]{article}
\usepackage{a4}
\usepackage{amsmath}
\usepackage{amssymb}
\usepackage{amsfonts}
\usepackage{amsthm}
\usepackage{graphicx}
\usepackage{color}
  \newtheorem{thm}{Theorem}[section]
 
 \newtheorem{prop}[thm]{Proposition}

 \newtheorem{lemma}[thm]{Lemma}
 \newtheorem{rem}[thm]{Remark}

\renewcommand{\t}{\theta}
\renewcommand{\a}{\alpha}
\renewcommand{\b}{\beta}

\newcommand{\la}{\lambda}
 \newcommand{\eps}{\varepsilon}
 
\newcommand{\C}{\mathbb{C}}
\newcommand{\R}{\mathbb{R}}
\renewcommand{\S}{\mathbb{S}}
\newcommand{\N}{\mathbb{N}}

\newcommand{\D}{\mathbb{D}}

 \title{Schoenberg's theorem for real and complex Hilbert spheres revisited}
 \author{Christian Berg, Ana P. Peron
  \; and Emilio Porcu}
 \date{\today}

\begin{document}

 \maketitle
 
 \begin{abstract} Schoenberg's theorem for the  complex Hilbert sphere proved by Christensen and Ressel in 1982 by Choquet theory is extended to the following result: Let $L$ denote a locally compact group and let $\overline{\D}$ denote the closed unit disc in the complex plane.  Continuous functions $f:\overline{\D}\times L\to \C$  such that $f(\xi\cdot\eta,u^{-1}v)$ is a positive definite kernel on the product of the unit sphere in $\ell_2(\C)$ and $L$ are characterized as the functions with a uniformly convergent expansion
$$
f(z,u)=\sum_{m,n=0}^\infty \varphi_{m,n}(u)z^m\overline{z}^n,
$$
where $\varphi_{m,n}$ is a double sequence of continuous positive definite functions on $L$ such that $\sum\varphi_{m,n}(e_L)<\infty$  ($e_L$ is the neutral element of $L$). It is shown how the coefficient functions $\varphi_{m,n}$ are obtained as limits  from  expansions for positive definite functions on finite dimensional complex spheres via a Rodrigues formula for disc polynomials.

Similar results are obtained for the real Hilbert sphere. 
 \end{abstract}



\section{Introduction and main results}  In his seminal paper,  Schoenberg \cite{S} introduced and characterized positive definite functions on spheres. The $d$-dimensional unit sphere of $\R^{d+1}$ is given as
\begin{equation*}\label{eq:sphere}
\mathbb S^d=\left\{x\in\mathbb R^{d+1}\mid \sum_{k=1}^{d+1} x_k^2=1\right\}, \;d\ge 1.
\end{equation*} 

For vectors $\xi,\eta$ belonging to $\S^d$, the scalar product $\xi\cdot\eta$ belongs to $[-1,1]$. 
By $\mathcal P(\S^d)$ we denote the set  of continuous functions $f:[-1,1]\to \R$ such that the kernel  $(\xi,\eta)\mapsto f(\xi\cdot\eta)$ is positive definite on $\S^d$ in the sense that  
for any $n\in\N$, arbitrary $\xi_1,\ldots,\xi_n\in\S^d$ and $c_1,\ldots,c_n\in\R$ one has
\begin{equation}\label{eq:pd}
\sum_{j,k=1}^n f(\xi_j\cdot\xi_k)c_jc_k\ge 0,
\end{equation}
i.e., the symmetric matrix $[f(\xi_j\cdot\xi_k)_{j,k=1}^n]$ is positive semidefinite.

In a general setting we recall that for an arbitrary non-empty set $X$, a kernel on $X$ is a function $k:X^2\to\C$. It is called a positive definite kernel on $X$ if for any $n\in\N$, any finite collection of points $x_1,\ldots,x_n\in X$ and numbers $c_1,\ldots,c_n \in \C$  one has
$$
\sum_{j,k=1}^n k(x_j,x_k)c_j\overline{c_k}\ge 0,
$$
i.e., the matrix $[k(x_j,x_k)_{j,k=1}^n]$ is hermitian and positive semidefinite. For a treatment of these concepts see e.g.  \cite{B:C:R}.  A positive definite kernel on $\S^d$ of the form $f(\xi\cdot\eta)$ is automatically real-valued by symmetry of the scalar product.

Let $L$ denote an arbitrary locally compact group written multiplicatively and with neutral element $e_L$. By $\mathcal P(L)$ we denote the set of continuous functions $f:L\to\C$ for which the kernel $(u,v)\mapsto f(u^{-1}v)$
is  positive definite on $L$. This class of functions is very important in the theory of unitary representations of $L$ on Hilbert spaces, see \cite{D},\cite{sasvari}.

Schoenberg's characterization of the class $\mathcal P(\S^d)$ was then extended by \cite{B:P}, who considered the class $\mathcal P(\S^d,L)$ of continuous functions $f:[-1,1]\times L\to\C$ such that the kernel
$((\xi,u),(\eta,v))\mapsto f(\xi\cdot\eta,u^{-1}v)$ is positive definite on $\S^d\times L$.

Their result is reported here for a self-contained exposition.

\begin{thm}\label{thm:BP1}{\rm (Theorem 3.3 in \cite{B:P})} Let $d\in\N$  and let $f:[-1,1]\times L\to \mathbb C$ be a continuous function. Then $f$ belongs to $\mathcal P(\S^d,L)$ 
 if and only if there exists
 a sequence of functions $(\varphi_{n,d})_{n\ge 0}$ from $\mathcal P(L) $ with $\sum_n \varphi_{n,d}(e_L)<\infty$ 
such that
\begin{equation}\label{eq:expand}
f(x,u)=\sum_{n=0}^\infty \varphi_{n,d}(u) c_n(d,x),\quad x\in[-1,1],\;u\in L.
\end{equation}
The above expansion is uniformly convergent for $(x,u)\in [-1,1]\times L$, and we have
\begin{equation}\label{eq:coef}
\varphi_{n,d}(u)=\frac{N_n(d)\sigma_{d-1}}{\sigma_d}\int_{-1}^1 f(x,u)c_n(d,x)(1-x^2)^{d/2-1}\, dx.
\end{equation}
\end{thm}

Here we have used the notation
\begin{equation}\label{eq:Geg}
c_n(d,x)=C_n^{(\la)}(x)/C_n^{(\la)}(1),\quad \la=(d-1)/2,\; d\ge 2,
\end{equation}
for the ultraspherical polynomials $c_n(d,x)$ as normalized Gegenbauer polynomials $C_n^{(\la)}(x)$ for the parameter $\la=(d-1)/2$, while $c_n(1,x)=T_n(x)$ are the Chebyshev polynomials, cf. \cite{A:A:R}, \cite{B:P}. For later use we recall that
\begin{equation}\label{eq:gegval1}
C_n^{(\la)}(1)=\frac{(2\la)_n}{n!}, \quad \la >0.
\end{equation}

 The constant $\sigma_d$ denotes the total mass of the surface measure $\omega_d$  on $\S^d$
\begin{equation}\label{eq:mass} 
\sigma_d=\omega_d(\S^d)=\frac{2\pi^{(d+1)/2}}{\Gamma((d+1)/2)}.
\end{equation}
Note that
$$
\frac{\sigma_{d-1}}{\sigma_d}\int_{-1}^1 (1-x^2)^{d/2-1}\,d x=1.
$$
Finally, $N_n(d)$ is the dimension of a space of spherical harmonics, cf. \cite[(11)]{B:P},\cite{M}, and is given by
\begin{equation}\label{eq:dim}
N_n(d)=\frac{(d)_{n-1}}{n!}(d+2n-1),\;n\ge 1,\quad N_0(d)=1.
\end{equation} 

Schoenberg's Theorem for $\mathcal P(\S^d)$ is the special case of the previous theorem, where the group $L=\{e_L\}$ is trivial. The functions in $\mathcal P(L)$ are then just non-negative constants.

The coefficient functions $\varphi_{n,d},n\ge 0$, of Theorem~\ref{thm:BP1} are called the $d$-Schoenberg  functions associated to $f$.

If we restrict the vectors $\xi_1,\ldots,\xi_n\in \S^d$  to lie on the subsphere
$\S^{d-1}$, identified with the equator of $\S^d$, we see that $\mathcal P(\S^d,L)\subseteq \mathcal P(\S^{d-1},L)$. 

We also consider
\begin{equation}\label{eq:infinity}
\mathcal P(\S^{\infty},L):=\bigcap_{d=1}^\infty \mathcal P(\S^d,L),
\end{equation} 
which is the set of continuous functions  $f:[-1,1]\times L\to \C$ such that the kernel 
\begin{equation}\label{eq:pdL}  
((\xi,u),(\eta,v))\mapsto  f(\xi\cdot\eta, u^{-1}v)
\end{equation}
 is positive definite  on $\S^d\times L$ for all $d\in\N$. We note in passing that the notation 
$\mathcal P(\S^{\infty},L)$ suggests an intrinsic definition using the real Hilbert sphere
$$
\S^\infty=\left\{(x_k)_{k\in\N}\in \R^\N \mid \sum_{k=1}^\infty x_k^2=1\right\},
$$
which is the unit sphere in the Hilbert sequence space $\ell_2(\R)$ of square summable real sequences.
The intrinsic definition of $\mathcal P(\S^{\infty},L)$ is as the set of continuous functions 
$f:[-1,1]\times L\to \C$ such that the kernel of Equation \eqref{eq:pdL} is positive definite on
$\S^\infty\times L$. This identification is made explicit in \cite{B:P}.

The following holds:

\begin{thm}\label{thm:BP2}{\rm (Theorem 3.10 in \cite{B:P})} Let $L$ denote a locally compact group and let 
$f:[-1,1]\times L\to \mathbb C$ be a continuous function. Then
$f$ belongs to $\mathcal P(\S^\infty,L)$ if and only if there exists
 a sequence $(\varphi_{n})_{n\ge 0}$ from $\mathcal P(L) $ with $\sum_n \varphi_{n}(e_L)<\infty$ 
such that
\begin{equation}\label{eq:expandp}
f(x,u)=\sum_{n=0}^\infty \varphi_{n}(u) x^n.
\end{equation}
The above expansion is uniformly convergent for $(x,u)\in [-1,1]\times L$.
\end{thm}

Schoenberg's Theorem for $\mathcal P(\S^\infty)$ (Theorem 2 in \cite{S}) is the special case of $L=\{e_L\}$, where
$\mathcal P(L)$ reduces to the non-negative constants.

Ziegel \cite{Z} discovered that $f\in\mathcal P(\S^d)$ is continuously differentiable of order $[(d-1)/2]$ on $]-1,1[$ and this was extended to $\mathcal P(\S^d,L)$ in \cite{B:P}. 

For $f\in\mathcal P(\S^\infty,L)$ we know that $f\in\mathcal P(\S^d,L)$ for all $d\in\N$ and we have expansions \eqref{eq:expand} of $f$ with $d$-Schoenberg functions $\varphi_{n,d}$ for each $d\in\N$. Furthermore, $f(\cdot,u)\in C^\infty(]-1,1[)$  for each $u\in L$ and
$$
\varphi_n(u)=\frac{1}{n!}\frac{\partial^n f(0,u)}{\partial x^n},\quad u\in L, n\ge 0.
$$

Using the smoothness theorem of Ziegel the following addition to Theorem~\ref{thm:BP2} was proved in \cite{B:P}:

\begin{equation}\label{eq:expandpp}
\lim_{d\to\infty} \varphi_{n,d}(u)=\varphi_n(u) ,\quad  u\in L, n\ge 0.
\end{equation}

Schoenberg did not have the special case of Equation \eqref{eq:expandpp}, where $L=\{e_L\}$, because he lacked the smoothness result of Ziegel.

In  \cite{B:P} the coefficient sequence $\varphi_n(u),n\ge 0$, was obtained  as an accumulation point of the sequence $(\varphi_{n,d}(u))_{d\ge 1}$ for $u,n$ fixed. The convergence in Equation \eqref{eq:expandpp} followed since it was possible to prove that the accumulation points were uniquely determined.  By this method it does not seem possible to prove that the convergence in Equation \eqref{eq:expandpp} is uniform for $u$ in compact subsets of $L$.   

In this paper we have a different approach which yields the local uniform convergence. It is based on  Rodrigues formula for the Gegenbauer polynomials. 

\begin{thm}\label{thm:main1} Let $f\in\mathcal P(\S^\infty,L)$ and let $\varphi_{n,d},n\ge 0,$ denote the $d$-Schoenberg functions associated to $f$. For each $n\in\N_0$, 
$$
\lim_{d\to\infty}\varphi_{n,d}(u)=\frac{1}{n!}\frac{\partial^n f(0,u)}{\partial x^n}
$$
uniformly for $u$ in compact subsets of $L$. The sequence of functions
$$
u\mapsto \frac{1}{n!}\frac{\partial^n f(0,u)}{\partial x^n}, \quad n\ge 0,
$$
belongs to $\mathcal P(L)$ and gives the coefficient sequence in Equation \eqref{eq:expandp}.
\end{thm}

The proof will be given in Section 2.

\bigskip

We are now going to explain  similar results for complex spheres.

The  complex unit sphere of (real) dimension $2q-1$ is given by
$$
\Omega_{2q}=\left\{z\in\C^{q} \mid ||z||^2=\sum_{k=1}^{q} |z_k|^2=1\right\},\qquad q\ge 1.
$$
Note that  $\Omega_{2q}$ is equal to $\S^{2q-1}$, if $\C^q$ is identified with $\R^{2q}$.
In the following we shall always assume that $q\ge 2$ because $\Omega_2=\S^1$ is an abelian group, and functions on this group are treated via Fourier series.

For vectors $\xi,\eta\in\Omega_{2q}$ the (hermitian) scalar product $\xi\cdot\eta$ belongs to the closed unit disc $\overline{\D}$, where the open disk $\D$ is defined as
$$
\D=\{z\in\C\mid |z|<1\}.
$$ 
By $\mathcal P(\Omega_{2q})$ we denote the set of continuous functions $f:\overline{\D}\to\C$ such that the kernel $(\xi,\eta)\mapsto f(\xi\cdot\eta)$ is positive definite on $\Omega_{2q}$. For a locally compact  group $L$ we denote by $\mathcal P(\Omega_{2q},L)$ the set of continuous functions $f:\overline{\D}\times L\to\C$ such that the kernel $((\xi,u),(\eta,v))\mapsto f(\xi\cdot\eta,u^{-1}v)$ is positive definite on $\Omega_{2q}\times L$. When $L=\{e_L\}$ is trivial, then $\mathcal P(\Omega_{2q},L)$ can be identified with  $\mathcal P(\Omega_{2q})$. 

In \cite{B:P:P} the authors proved the following result, which extended a result by Menegatto and Peron \cite[Theorem 4.2]{M:P} for the case $\mathcal P(\Omega_{2q})$. 

\begin{thm}\label{thm:MP} {\rm (Theorem 6.1 in \cite{B:P:P})} Let $q\ge 2$ and let $f:\overline{\D}\times L\to\C$ be a continuous function. Then $f$ belongs to $\mathcal P(\Omega_{2q},L)$ if and only if there exists a double sequence of functions $(\varphi_{m,n}^{q-2})_{m,n\ge 0}$ from $\mathcal P(L)$ with
$$
\sum_{m,n=0}^\infty \varphi_{m,n}^{q-2}(e_L)<\infty
$$
 such that
\begin{equation}\label{eq:expandcp}
f(z,u)=\sum_{m,n=0}^\infty \varphi_{m,n}^{q-2}(u) R_{m,n}^{q-2}(z),\quad z\in\overline{\D},u\in L.
\end{equation}
The above expansion is uniformly convergent on $\overline{\D}\times L$, and for $u\in L$ we have
\begin{equation}\label{eq:coefcp}
\varphi_{m,n}^{q-2}(u)=N(q;m,n)\frac{q-1}{\pi}\int_{0}^1\int_0^{2\pi}  f(re^{i\t},u)\overline{R^{q-2}_{m,n}(re^{i\t})}r(1-r^2)^{q-2}\,dr \, d\t.
\end{equation}
\end{thm} 

Here 
$$
N(q;m,n)=\frac{m+n+q-1}{q-1}\binom{m+q-2}{m}\binom{n+q-2}{n}
$$
is the dimension of a certain finite-dimensional space, see \cite{K2},\cite{V:K}.
The functions    $R^{q-2}_{m,n}(z)$ belong to the class of disc polynomials given in \cite{K2} for $\a>-1$ as
\begin{eqnarray*}\label{eq:discpol}
R^\a_{m,n}(r e^{i\t})=r^{|m-n|}e^{i(m-n)\t}R^{(\a,|m-n|)}_{\min(m,n)}(2r^2-1),\quad 0\le r\le 1,\;0\le \t<2\pi
\end{eqnarray*}
and 
\begin{equation}\label{eq:Jacobinorm}
R^{(\a,\b)}_k(x)=P^{(\a,\b)}_k(x)/P^{(\a,\b)}_k(1),\quad \a,\b>-1,\;k\in \N_0
\end{equation}
are normalized Jacobi polynomials $P^{(\a,\b)}_k$,  cf. \cite{A:A:R}.

See \cite{Wu} for other expressions and properties of the disc polynomials.

Like the case of real spheres we have
$$
\mathcal P(\Omega_{2(q+1)},L) \subseteq \mathcal P(\Omega_{2q},L),
$$
and we consider the set
\begin{equation}\label{eq:cpsphereinfty}
\mathcal P(\Omega_{\infty},L):=\bigcap_{q=2}^\infty \mathcal P(\Omega_{2q},L),
\end{equation}
which can be identified with the set of continuous functions $f:\overline{\D}\times L\to\C$ such that the kernel $((\xi,u),(\eta,v))\mapsto f(\xi\cdot \eta,u^{-1}v)$ is positive definite on 
$\Omega_\infty\times L$, where
$$
\Omega_\infty=\left\{(z_k)_{k\in\N}\in \C^\N \mid \sum_{k=1}^\infty |z_k|^2=1 \right\}
$$
is the unit sphere in the Hilbert sequence space $\ell_2(\C)$ of square summable complex sequences.

The second purpose of this paper is to give a proof of the following result:

\begin{thm}\label{thm:Schinftycp} Let $L$ denote a locally compact group and let $f:\overline{\D}\times L\to\C$ be a continuous function. Then $f$ belongs to $\mathcal P(\Omega_{\infty},L)$ if and only if there  exists a double sequence of functions $(\varphi_{m,n})_{m,n\ge 0}$ from $\mathcal P(L)$ with
$$
\sum_{m,n=0}^\infty \varphi_{m,n}(e_L)<\infty
$$
such that
\begin{equation}\label{eq:expandcpinfty}
f(z,u)=\sum_{m,n=0}^\infty \varphi_{m,n}(u)z^m\overline{z}^n,\quad z\in\overline{\D},u\in L.
\end{equation}
The series in Equation \eqref{eq:expandcpinfty} is uniformly convergent on $\overline{\D}\times L
$.
\end{thm}

When $L=\{e_L\}$ is  trivial and $\mathcal P(L)$ reduces to the set of non-negative constants, Theorem~\ref{thm:Schinftycp} is a result of Christensen and Ressel \cite{C:R}, also treated in \cite[Chapter 5.4]{B:C:R}. 

For a function $f\in \mathcal P(\Omega_{\infty},L)$ we have an expansion
\eqref{eq:expandcp} for each $q\ge 2$ due to \eqref{eq:cpsphereinfty}. The connection to 
\eqref{eq:expandcpinfty} is given by the following result:

\begin{thm}\label{thm:main2} For the uniquely determined coefficient functions $\varphi_{m,n}^{q-2},\varphi_{m,n}\in\mathcal P(L)$ from \eqref{eq:expandcp} and \eqref{eq:expandcpinfty}, we have for  $m,n\in\N_0$
$$
\lim_{q\to\infty}\varphi_{m,n}^{q-2}(u)=\varphi_{m,n}(u)=\frac{1}{m!n!}\frac{\partial^{m+n} f(0,u)}{\partial z^m \partial\overline{z}^n }
$$
uniformly for $u$ in compact subsets of $L$.
\end{thm}

The proofs of Theorem~\ref{thm:Schinftycp} and \ref{thm:main2} will be given in Section 3.

\section{Proofs in the case of the real Hilbert sphere}

 We need the following sharpening of Proposition 3.8 in \cite{B:P}, which is inspired by results of Ziegel \cite{Z}.

 For functions $F:[-1,1]\times L\to\C$ we denote
$$
||F||=\sup\{|F(x,u)|\mid x\in[-1,1],u\in L\}\le \infty.
$$
Note that  if $f\in\mathcal P(\S^d,L)$ then $||f||=f(1,e_L)<\infty$.

\begin{prop}\label{thm:Z1} Let $d\in\N$ and suppose that $f\in\mathcal P(\S^{d+2},L)$. Then $f(\cdot,u)$ is continuously differentiable with respect to $x$ in $]-1,1[$ and
$(1-x^2)\frac{\partial f(x,u)}{\partial x}$ extends to a continuous function on $[-1,1]\times L$
such that
\begin{equation}\label{eq:eq1}
(1-x^2)\frac{\partial f(x,u)}{\partial x}=f_1(x,u)-f_2(x,u),\quad (x,u)\in[-1,1]\times L
\end{equation}
for functions $f_i\in\mathcal P(\S^d,L)$ satisfying
\begin{equation}\label{eq:norm}
||f_i||\le d||f||, \qquad i=1,2.
\end{equation}
\end{prop}

\begin{proof} Let us first assume $d\ge 2$. By the proof of Proposition 3.8 in \cite{B:P} we have
\eqref{eq:eq1} for $(x,u)\in \left]-1,1\right[\times L$, where
 $$
f_1(x,u)=d\sum_{n=0}^\infty\frac{(2n+d-1)(n+1)}{(2n+d+1)(n+d-1)}\varphi_{n+1,d}(u)c_n(d,x)
$$ 
and
$$
f_2(x,u)=d\sum_{n=2}^\infty \frac{n-1}{n+d-1}\varphi_{n-1,d+2}(u)c_n(d,x).
$$
These formulas show that $f_1,f_2\in\mathcal P(\S^d,L)$ and that
\begin{eqnarray*}
||f_1||&=&f_1(1,e_L)\\ &=& d\sum_{n=0}^\infty\frac{(2n+d-1)(n+1)}{(2n+d+1)(n+d-1)}\varphi_{n+1,d}(e_L)\\
&\le& d\sum_{n=0}^\infty \varphi_{n+1,d}(e_L)=d\sum_{n=1}^\infty \varphi_{n,d}(e_L)\le d||f||,
\end{eqnarray*}
and
$$
||f_2||=f_2(1,e_L)=d\sum_{n=2}^\infty\frac{n-1}{n-1+d}\varphi_{n-1,d+2}(e_L)
\le d\sum_{n=1}^\infty \varphi_{n,d+2}(e_L)\le d||f||.
$$
This also
shows that the left-hand side of Equation \eqref{eq:eq1} is continuous
on $[-1,1]\times L$.

For $d=1$ Equation \eqref{eq:eq1} holds again, now with
$$
f_1(x,u)=\frac{1}{2}\varphi_{1,1}(u)c_0(1,x)+\sum_{n=1}^\infty \varphi_{n+1,1}(u)c_n(1,x)
$$
and
$$
f_2(x,u)=\sum_{n=2}^\infty \frac{n-1}{n}\varphi_{n-1,3}(u)c_n(1,x).
$$
This shows that \eqref{eq:norm} holds also in this case.
\end{proof}

Let $\mathcal E_d$ denote the subspace of continuous functions $F:[-1,1]\times L\to\C$ spanned by functions of the form $p(x)f(x,u)$, where $p$ is a  polynomial with complex coefficients and $f\in\mathcal P(\S^d,L)$. By Proposition~\ref{thm:Z1} we see that $(1-x^2)\partial/\partial x$ maps $\mathcal E_{d+2}$ into $\mathcal E_d$.

\begin{prop}\label{thm:Z2} Let $d,n\in\N$ and assume that $f\in\mathcal P(\S^{d+2n},L)$. Then 
$f(\cdot,u)\in C^n(]-1,1[)$ for $u\in L$ and for $k\le n$ we have
\begin{equation}\label{eq:E2n}
(1-x^2)^k\frac{\partial^k f(x,u)}{\partial x^k}\in\mathcal E_{d+2(n-k)}.
\end{equation}
In particular the function in Equation \eqref{eq:E2n}
has a continuous extension to $[-1,1]\times L$.
\end{prop}

\begin{proof} It follows by Proposition~\ref{thm:Z1} that $f(\cdot,u)\in C^n(]-1,1[)$  for $u\in L$.

We prove \eqref{eq:E2n} by induction in $k$, and it certainly holds for $k=1$ by  Proposition~\ref{thm:Z1}.
 
Suppose \eqref{eq:E2n} holds for $k<n$. Then the function in \eqref{eq:E2n} is  differentiable
for $-1<x<1$  and differentiation and multiplication with $1-x^2$ shows that
\begin{eqnarray*}
 - 2kx(1-x^2)^{k}\frac{\partial^k f(x,u)}{\partial x^k}+(1-x^2)^{k+1} \frac{\partial^{k+1} f(x,u)}{\partial x^{k+1}}\in\mathcal E_{d+2(n-k-1)}.
\end{eqnarray*}
Using
$$
2kx(1-x^2)^{k}\frac{\partial^k f(x,u)}{\partial x^k}\in\mathcal E_{d+2(n-k)}\subseteq \mathcal E_{d+2(n-k-1)},
$$
we see that
$$
(1-x^2)^{k+1} \frac{\partial^{k+1} f(x,u)}{\partial x^{k+1}}\in\mathcal E_{d+2(n-k-1)}.
$$
\end{proof}

In the next proposition we prove the weak convergence of a certain family $(\tau_{\la})$ of measures introduced below. This convergence is decisive for the proof of our main Theorem~\ref{thm:main1}.

For $\la>-1$ define the probability measure $\tau_{\la}$ on $[-1,1]$  by
\begin{equation}\label{eq:tau}
\tau_\la=B(\la+1,1/2)^{-1}(1-x^2)^{\la}\,dx, 
\end{equation}
 where $B$ is the Beta-function.

The set $C([-1,1])$ of continuous functions $f:[-1,1]\to\ \C$ is a Banach space under the uniform norm $||f||_\infty=\sup_{x\in[-1,1]}|f(x)|$.

\begin{prop}\label{thm:Poin} Let $\mathcal F\subset C([-1,1])$ be a set of continuous functions on $[-1,1]$ such that
\begin{enumerate}
\item[{\rm (i)}] $\mathcal F$ is bounded, i.e., $\sup_{f\in\mathcal F}||f||_\infty<\infty$,
\item[{\rm (ii)}] $\mathcal F$ is equicontinuous at $x=0$, i.e., for every $\eps>0$ there exists $0<\delta<1$ such that $|f(x)-f(0)|\le \eps$ for  all $f\in\mathcal F$ and all  real $x$ with $|x|\le \delta$.
\end{enumerate}

Then $\lim_{\la\to\infty} \int f\,d\tau_{\la}=f(0)$, uniformly for $f\in\mathcal F$.

In particular, $\lim_{\la\to\infty}\tau_\la=\delta_0$ weakly, where $\delta_0$ denotes the Dirac measure concentrated at $0$. 
\end{prop}

\begin{proof} For any $0<\delta<1$ and $f\in\mathcal F$ we have
$$
\int f\,d\tau_\la -f(0)=\int_{-\delta}^\delta \bigg ( f(x)-f(0) \bigg ) \,d\tau_\la(x)+\int_{\delta\le |x|\le 1}
\bigg ( f(x)-f(0) \bigg ) \,d\tau_\la(x).
$$
Using $|f(x)-f(0)|\le 2||f||_\infty$, we get for $\la>0$
\begin{eqnarray*}
\bigg |\int f\,d\tau_\la-f(0)\bigg | &\le& \sup_{|x|\le\delta}\bigg | f(x)-f(0)\bigg |+\frac{2||f||_\infty}{B(\la+1,1/2)}\int_{\delta\le |x|\le 1}(1-x^2)^{\la}\,dx\\
&\le& \sup_{|x|\le\delta}\bigg | f(x)-f(0)\bigg | + \frac{4||f||_\infty(1-\delta)}{B(\la + 1,1/2)}(1-\delta^2)^{\la}.
\end{eqnarray*}

For given $\eps>0$ we first choose $0<\delta<1$ so that by (ii)
$$
|f(x)-f(0)|\le \eps/2,\quad\mbox{for all}\; |x|\le\delta,\;f\in\mathcal F.
$$
By Stirling's formula
$$
B(\la+1,1/2)^{-1}\sim \pi^{-1/2}\la^{1/2},\quad \la\to\infty,
$$
and $\la^{1/2}(1-\delta^2)^\la\to 0$ for $\la\to\infty$. Therefore, and using (i),
$$ 
 \sup_{f\in\mathcal F}||f||_\infty \frac{4(1-\delta)}{B(\la + 1,1/2)}(1-\delta^2)^{\la}<\eps/2
$$
for $\la\ge \Lambda_0$, where $\Lambda_0$ is sufficiently large. This shows that
$$
\sup_{f\in\mathcal F}\left|\int f\;d\tau_\la -f(0)\right|\le\eps,\quad \la\ge\Lambda_0.
 $$ 
\end{proof}

\noindent {\bf Proof of Theorem~\ref{thm:main1}:}

It is known that the Gegenbauer polynomials $C^{(\la)}_n(x)$ satisfy the Rodrigues formula, cf. \cite[(6.6.14)]{A:A:R}
$$
C^{(\la)}_n(x)=\frac{(-2)^n(\la)_n}{n!(n+2\la)_n}(1-x^2)^{1/2-\la}\frac{d^n}{d x^n}(1-x^2)^{n+\la-1/2}.
$$
For the normalized ultraspherical polynomials $c_n(d,x)$ given by \eqref{eq:Geg}, 
the Rodrigues formula reads

\begin{equation}\label{eq:Rod}
c_n(d,x)=\frac{(-1)^n}{2^n(d/2)_n}(1-x^2)^{1-d/2}\frac{d^n}{d x^n}(1-x^2)^{n+d/2-1}.
\end{equation}
Inserting this in Equation \eqref{eq:coef} we get
$$
\varphi_{n,d}(u)=\frac{N_n(d)\sigma_{d-1}}{\sigma_d}\frac{(-1)^n}{2^n(d/2)_n}\int_{-1}^1
f(x,u)\frac{d^n}{d x^n}(1-x^2)^{n+d/2-1}\,dx.
$$
We now make use of $n$ integrations by parts to get
$$
\varphi_{n,d}(u)=\frac{N_n(d)\sigma_{d-1}}{\sigma_d}\frac{1}{2^n(d/2)_n}\int_{-1}^1
\frac{\partial^{n}f(x,u)}{\partial x^n} (1-x^2)^{n+d/2-1}\,dx,
$$
because the boundary terms
$$
\frac{\partial^{k}f(x,u)}{\partial x^k}\frac{d^{n-k-1}}{d x^{n-k-1}}(1-x^2)^{n+d/2-1},\;k=0,1,\ldots,n-1
$$
vanish for $x=\pm 1$ by Proposition~\ref{thm:Z2}.
 In fact, 
$$
\frac{d^{n-k-1}}{d x^{n-k-1}}(1-x^2)^{n+d/2-1}=(1-x^2)^{k+d/2}R_k(x)
$$
for some polynomial $R_k(x)$ and
$$
 (1-x^2)^{k}\frac{\partial^{k} f(x,u)}{\partial x^k}
$$
has finite values while $(1-x^2)^{d/2}R_k(x)$ vanishes for $x=\pm 1$.

Using the measure \eqref{eq:tau} with $\la=d/2-1$, we find
$$
\varphi_{n,d}(u)=\frac{N_n(d)}{2^n(d/2)_n}\int_{-1}^1 (1-x^2)^n\frac{\partial^{n}f(x,u)}{\partial x^n}\,d\tau_{d/2-1}(x),
$$
and we note that
$$
\frac{N_n(d)}{2^n(d/2)_n}=\frac{1}{n!}\frac{(d)_{n-1}(d+2n-1)}{2^n(d/2)_n}\to\frac{1}{n!}\;\;\mbox{for}\;\; d\to\infty.
$$
By Proposition~\ref{thm:Poin} we then get that
\begin{equation}\label{eq:final}
\varphi_{n,d}(u)\to \frac{1}{n!}\left[ (1-x^2)^n\frac{\partial^{n}f(x,u)}{\partial x^n}\right]_{x=0}=\frac{1}{n!}\frac{\partial^{n}f(0,u)}{\partial x^n}.
\end{equation}
Given a compact set $K$ in $L$ the family
$$
\mathcal F:=\left\{ x\mapsto (1-x^2)^n\frac{\partial^{n}f(x,u)}{\partial x^n}\mid u\in K\right\}
$$
satisfies the conditions of Proposition~\ref{thm:Poin}, so the convergence in \eqref{eq:final} is uniform for $u$ in compact subsets of $L$.

This also implies that $u\mapsto \frac{1}{n!}\frac{\partial^{n}f(0,u)}{\partial x^n}$ belongs to $\mathcal P(L)$ and is the coefficient $\varphi_n(\cdot)$ of the power series in \eqref{eq:expandp}.
\hfill$\square$

\section{Proofs in the case of the complex Hilbert sphere}

Let us first consider a function $f\in\mathcal P(\Omega_{2q},L)$. Then we know that
$$
\overline{f(z,u)}=f(\overline{z},u^{-1}),\quad |f(z,u)|\le f(1,e),\quad z\in\overline{\D}, u\in L.
$$

To $f$ and to elements $u_1,\ldots,u_n\in L$ and numbers $c_1,\ldots,c_n\in \C$ we define a new function $F:\overline{\D}\to \C$ by
\begin{equation}\label{eq:sum1}
F(z)=\sum_{j,k=1}^n f(z,u_j^{-1}u_k)c_j\overline{c_k}.
\end{equation}
It is easy to see that $F(\overline{z})=\overline{F(z)}$, but in fact, this follows from the more general result inspired by \cite{G:M} and which can be stated as:
\begin{lemma}\label{thm:technical}
For any $f$ in $\mathcal P(\Omega_{2q},L)$, the function $F$ in \eqref{eq:sum1} belongs to $\mathcal P(\Omega_{2q})$. 
\end{lemma}

\begin{proof} Let $\xi_1,\ldots,\xi_m\in\Omega_{2q}$ and $d_1,\ldots,d_m\in\C$ be arbitrary. We
shall prove that $S\ge 0$, where 
$$
S:=\sum_{\mu,\nu=1}^m F(\xi_\mu\cdot\xi_\nu)d_\mu\overline{d_\nu}.
$$
However, 
$$
S=\sum_{\mu,\nu=1}^m \sum_{j,k=1}^n f(\xi_\mu\cdot\xi_\nu,u_j^{-1}u_k)c_j\overline{c_k}
d_\mu\overline{d_\nu} \ge 0,
$$
because it is "a sum" belonging to the finite family of $mn$ elements from $\Omega_{2q}\times L$
$$
(\xi_1,u_1),\ldots,(\xi_1,u_n), (\xi_2,u_1),\ldots,(\xi_2,u_n),\ldots, (\xi_m,u_1),\ldots,(\xi_m,u_n)
$$ 
together with the family of scalars
$$
d_1c_1,\ldots, d_1c_n, d_2c_1,\ldots, d_2c_n,\ldots,  d_mc_1,\ldots, d_mc_n.
$$
\end{proof}

\noindent {\bf Proof of Theorem~\ref{thm:Schinftycp}:}

It is easy to see that $(\xi,\eta)\mapsto \xi\cdot\eta$ is a positive definite kernel on $\Omega_{2q}$. By the Schur product theorem for positive definite kernels, cf. \cite[Theorem 3.1.12]{B:C:R}, we see that $z^m\overline{z}^n$ belongs to $\mathcal P(\Omega_{2q})$ for $q\ge 2$ and $m,n\ge 0$. It is therefore elementary that any function of the form \eqref{eq:expandcpinfty}
with $\varphi_{m,n}\in\mathcal P(L)$ satisfying
$$
\sum_{m,n=0}^\infty \varphi_{m,n}(e_L)<\infty,
$$
belongs to $\mathcal P(\Omega_{\infty},L)$, which was  defined in \eqref{eq:cpsphereinfty}.

If we start with a continuous function $f:\overline{\D}\times L\to \C$ belonging to
$\mathcal P(\Omega_\infty,L)$,
then $F$ defined by \eqref{eq:sum1} belongs to 
$$
\mathcal P(\Omega_\infty)= \cap_{q=2}^\infty \mathcal P(\Omega_{2q})
$$
 by Lemma~\ref{thm:technical}. Using a theorem due to Christensen and Ressel, see \cite{C:R}, it can be written as
$$
F(z)=\sum_{m,n=0}^\infty a_{m,n} z^m\overline{z}^n,
$$
where $a_{m,n}\ge 0$ are uniquely determined by $F$ and satisfy $\sum a_{m,n}<\infty.$

We now use the special case of \eqref{eq:sum1} with $n=2$, $u_1=e_L, u_2=u$, $c_1=1, c_2=c$, so $F=F_{u,c}$ takes the form
\begin{equation}\label{eq:sum2}
F_{u,c}(z)=f(z,e_L)(1+|c|^2)+f(z,u)\overline{c} +f(z,u^{-1})c.
\end{equation}
For all $u\in L, c\in\C$ there exist 
$a_{m,n}(u,c)\ge 0$ with $\sum a_{m,n}(u,c)<\infty$ such that
$$
F_{u,c}(z)=\sum_{m,n=0}^\infty a_{m,n}(u,c)z^m\overline{z}^n,\quad z\in\overline{\D}.
$$
Letting $c=1,-1,i$ we obtain
\begin{eqnarray*}
F_{u,1}(z)&=& 2f(z,e_L)+f(z,u)+f(z,u^{-1})=\sum_{m,n=0}^\infty a_{m,n}(u,1)z^m\overline{z}^n,\\
F_{u,-1}(z)&=& 2f(z,e_L)-f(z,u)-f(z,u^{-1})=\sum_{m,n=0}^\infty a_{m,n}(u,-1)z^m\overline{z}^n,\\
F_{u,i}(z)&=& 2f(z,e_L) - i f(z,u)+ i f(z,u^{-1})=\sum_{m,n=0}^\infty a_{m,n}(u,i)z^m\overline{z}^n.
\end{eqnarray*}

This gives
$$
\frac{1-i}{4}F_{u,1}(z)-\frac{1+i}{4}F_{u,-1}(z)+\frac{i}{2}F_{u,i}(z)=f(z,u)=\sum_{m,n=0}^\infty \varphi_{m,n}(u)z^m\overline{z}^n,
$$ 
where
$$
\varphi_{m,n}(u):=\frac{1-i}{4}a_{m,n}(u,1)-\frac{1+i}{4}a_{m,n}(u,-1)+\frac{i}{2}a_{m,n}(u,i).
$$
 That $\varphi_{m,n}\in\mathcal P(L)$ can be seen as in \cite{G:M}, or we can use that necessarily
$$
\varphi_{m,n}(u)=\frac{1}{m!n!}\frac{\partial^{m+n}f(0,u)}{\partial z^m \partial\overline{z}^n},
$$
and that
\begin{equation}\label{eq:powercp}
\lim_{q\to\infty}\varphi_{m,n}^{q-2}(u)=\frac{1}{m!n!}\frac{\partial^{m+n} f(0,u)}{\partial z^m \partial\overline{z}^n }
\end{equation}
uniformly for $u$ in compact subsets of $L$ as stated in Theorem~\ref{thm:main2} .
 Formula \eqref{eq:powercp} proves that the functions on the right-hand side belong to $\mathcal P(L)$.
\hfill$\square$

\medskip

As preparation for the proof of Theorem~\ref{thm:main2} we shall discuss smoothness of functions from $\mathcal P(\Omega_{2q},L)$.

The smoothness results of Ziegel \cite{Z} for functions in $\mathcal P(\S^d)$ have been extended to functions in $\mathcal P(\Omega_{2q})$  in a paper by Menegatto, see  \cite{Me}. This extension required new ideas, while  a further extension to functions in $\mathcal P(\Omega_{2q},L)$ follows the same lines as in \cite{Me}, so we shall just give the results with a few indications.

For $f\in\mathcal P(\Omega_{2q+2},L)\subseteq\mathcal P(\Omega_{2q},L)$ we have the expansions, cf. Theorem~\ref{thm:MP},
$$
f(z,u)=\sum_{m,n=0}^\infty \varphi_{m,n}^{q-1}(u) R_{m,n}^{q-1}(z)=\sum_{m,n=0}^\infty \varphi_{m,n}^{q-2}(u) R_{m,n}^{q-2}(z),\quad z\in\overline{\D},u\in L,
$$
and the coefficient functions are related in the following way:

\begin{prop}\label{thm:relation-coeff}
Let $q\geq2$. If $f\in\mathcal P(\Omega_{2q+2},L)$, then for $m,n\ge 0$ and $u\in L$
\begin{equation}\label{eq:coef-func}
\varphi_{m,n}^{q-1}(u) = \frac{(m+q-1)(n+q-1)}{(q-1)(m+n+q-1)}\varphi_{m,n}^{q-2}(u) - \frac{(m+1)(n+1)}{(q-1)(m+n+q+1)}\varphi_{m+1,n+1}^{q-2}(u).
\end{equation}
In particular,
\begin{equation}\label{eq:coef-inq}
\varphi_{m,n}^{q-2}(e_L) \geq \frac{(m+1)(n+1)(m+n+q-1)}{(m+q-1)(n+q-1)(m+n+q+1)}\varphi_{m+1,n+1}^{q-2}(e_L).
\end{equation}
\end{prop}

\begin{proof} For the first part we can use the same technique as in the proof of  \cite[Proposition 4.1]{Me}. The second part follows from the fact  that $\varphi_{m,n}^{q-1}(e_L)\ge 0$ and from the first part.
\end{proof}

\begin{prop}\label{thm:lim_coeff}
Let $q\geq 2$. If $f\in\mathcal P(\Omega_{2q+2},L)$, then for each $u\in L$ fixed
$$
\lim_{M,N\to\infty}\sum_{n=1}^{N+1}\frac{M(n+q-2)}{M+n+q-2}\varphi_{M,n-1}^{q-2}(u) = 0
$$
and
$$
\lim_{M,N\to\infty}\sum_{m=1}^{M-1}\frac{m(N+q-1)}{m+N+q-1}\varphi_{m,N}^{q-2}(u) = 0.
$$
Both limits are uniform with respect to $u\in L$.
\end{prop}

\begin{proof} Define
$$
A_{M,N}:=\sum_{n=1}^{N+1}\frac{M(n+q-2)}{M+n+q-2}\varphi_{M,n-1}^{q-2}(u).
$$
Then
$$
|A_{M,N}| \leq \sum_{n=1}^\infty\frac{M(n+q-2)}{M+n+q-2}|\varphi_{M,n-1}^{q-2}(u)| \leq 
\sum_{n=1}^\infty\frac{M(n+q-2)}{M+n+q-2}\varphi_{M,n-1}^{q-2}(e_L).
$$
Define 
$$
c_M := \sum_{n=1}^\infty\frac{n+q-2}{M+n+q-2}\varphi_{M,n-1}^{q-2}(e_L), \quad M=1,2,\ldots.
$$
We have $0\leq c_M <\infty$ and
$$
 \sum_{M=1}^\infty c_M = \sum_{M=1}^\infty \sum_{n=1}^\infty\frac{n+q-2}{M+n+q-2}\varphi_{M,n-1}^{q-2}(e_L) <\infty,
$$
because 
$$
\frac{n+q-2}{M+n+q-2} \leq1, \quad   \sum_{m,n=0}^\infty\varphi_{m,n}^{q-2}(e_L) <\infty,
$$
and then we can use Lemma 3.2 in \cite{Me}.

Since $0\leq |A_{M,N}| \leq Mc_M$ for all $M,N$, we have
$$
\lim_{M,N\to\infty}A_{M,N}=0
$$ 
provided  $\lim_{M\to\infty}Mc_M=0$. To see this we get from \eqref{eq:coef-inq} with
$m=M$ and $n$ replaced by $n-1$

\begin{eqnarray*}
c_M &\ge& \sum_{n=1}^\infty\frac{(M+1)n}{(M+q-1)(M+n+q)}\varphi_{M+1,n}^{q-2}(e_L) \\
&=& \frac{M+1}{M+q-1}\sum_{n=1}^\infty\frac{n+q-2}{M+1+n+q-2}\left(1-\frac{q-1}{n+q-2}\right)
\varphi_{M+1,n-1}^{q-2}(e_L)\\
&=& \frac{M+1}{M+q-1}c_{M+1} - \frac{M+1}{M+q-1}\sum_{n=1}^\infty\frac{q-1}{M+n+q-1}\varphi_{M+1,n-1}^{q-2}(e_L)\\
&\ge&\frac{M+1}{M+q-1}c_{M+1} - \frac{(M+1)(q-1)}{(M+q-1)^2}\sum_{n=1}^\infty \varphi_{M+1,n-1}^{q-2}(e_L).
\end{eqnarray*}
Now, $\lim_{M\to\infty}Mc_M=0$ follows as in \cite{Me}.
\end{proof}

In analogy with Theorem 1.1 in \cite{Me} we have:

\begin{prop}\label{thm:Zcp} Let $q\ge 2$ and assume that $f\in\mathcal P(\Omega_{2q+2},L)$. Then $f(\cdot,u)$ is differentiable with respect to $z$ and $\overline{z}$ in $\D$ and there exist functions $f_i\in\mathcal P(\Omega_{2q},L), i=1,2,3,4$ such that for $(z,u)\in \D\times L$
\begin{eqnarray}\label{eq:parz}
(1-|z|^2)\frac{\partial f(z,u)}{\partial z} &=& f_1(z,u)-f_2(z,u)\\ \label{eq:parbarz} 
(1-|z|^2)\frac{\partial f(z,u)}{\partial \overline{z}} &=& f_3(z,u)-f_4(z,u). 
\end{eqnarray}
In particular, the two functions to the left in \eqref{eq:parz} and in \eqref{eq:parbarz} have continuous extensions to $\overline{\D}\times L$.
\end{prop}   

Let $\mathcal G_{2q}$ denote the  subspace of continuous functions $F:\overline{\D}\times L\to\C$  spanned by functions of the form $p(z,\overline{z})f(z,u)$, where $p$ is a polynomial in $z$ and $\overline{z}$ with complex coefficients and $f\in\mathcal P(\Omega_{2q},L)$.

By Proposition~\ref{thm:Zcp} we see that $(1-|z|^2)\partial/\partial z$ and $(1-|z|^2)\partial/\partial \overline{z}$ maps $\mathcal G_{2q+2}$ into $\mathcal G_{2q}$.

\begin{prop}\label{thm:Z2cp}  Let $q\ge 2$ and assume that $f\in\mathcal P(\Omega_{2q+2n},L)$ for $n\ge 1$. Then $f(\cdot,u)$ is $n$ times differentiable with respect to $z$ and $\overline{z}$ in $\D$ and for $r+s\le n$
\begin{equation}\label{eq:Z2cp}
(1-|z|^2)^{r+s}\frac{\partial^{r+s}f(z,u)}{\partial z^r \partial\overline{z}^s}\in\mathcal G_{2q+2(n-r-s)}.
\end{equation}
In particular, the function in Equation \eqref{eq:Z2cp} has a continuous extension to $\overline{\D}\times L$.
\end{prop}

\begin{proof}  We prove \eqref{eq:Z2cp} by induction  in $r+s$.

It certainly holds for $r+s=1$ by Proposition~\ref{thm:Zcp}.
Assume that it holds for $r+s<n$. 
Differentiating the function in \eqref{eq:Z2cp} with respect to $z$ and multiplication with $1-|z|^2$ shows that
\begin{eqnarray*}
(1-|z|^2)^{r+s+1}\frac{\partial^{r+s+1}f(z,u)}{\partial z^{r+1} \partial\overline{z}^s}
-(r+s)\overline{z}(1-|z|^2)^{r+s}\frac{\partial^{r+s}f(z,u)}{\partial z^{r} \partial\overline{z}^s}
\end{eqnarray*}
belongs to $\mathcal G_{2q+2(n-r-s-1)}$, and since
$$
(r+s)\overline{z}(1-|z|^2)^{r+s}\frac{\partial^{r+s}f(z,u)}{\partial z^{r} \partial\overline{z}^s}\in \mathcal G_{2q+2(n-r-s)}\subseteq
\mathcal G_{2q+2(n-r-s-1)},
$$
we see that
$$
(1-|z|^2)^{r+s+1}\frac{\partial^{r+s+1}f(z,u)}{\partial z^{r+1} \partial\overline{z}^s}\in\mathcal G_{2q+2(n-r-s-1)}.
$$ 
Differentiating  the function in \eqref{eq:Z2cp} with respect to $\overline{z}$ and multiplying with $(1-|z|^2)$ gives that
$$
(1-|z|^2)^{r+s+1}\frac{\partial^{r+s+1}f(z,u)}{\partial z^{r} \partial\overline{z}^{s+1}}\in\mathcal G_{2q+2(n-r-s-1)}.
$$ 
\end{proof}

In the next proposition we prove the weak convergence of a certain family $(\nu_{\a})$ of measures introduced below. This convergence is decisive for the proof of  Theorem~\ref{thm:main2}.

Let $\nu_\a,\a>-1,$ denote the probability measure on $\overline{\D}$ given by
\begin{equation}\label{eq:nual}
\nu_\a=\frac{\a+1}{\pi}(1-x^2-y^2)^\a  \,dx\,dy,\quad x^2+y^2<1,
\end{equation} 
and in polar coordinates the expression is
$$
\nu_\a=\frac{\a+1}{\pi}(1-r^2)^\a r\,dr\,d\t,\quad 0\le r< 1, 0\le \t < 2\pi.
$$

The set $C(\overline{\D})$ of continuous functions $f:\overline{\D}\to\ \C$ is a Banach space under the uniform norm $||f||_\infty=\sup_{z\in\overline{\D}}|f(z)|$.

\begin{prop}\label{thm:Poincp} Let $\mathcal F\subset C(\overline{\D})$ be a set of continuous functions on $\overline{\D}$ such that
\begin{enumerate}
\item[{\rm (i)}] $\mathcal F$ is bounded, i.e., $\sup_{f\in\mathcal F}||f||_\infty<\infty$,
\item[{\rm (ii)}] $\mathcal F$ is equicontinuous at $z=0$, i.e., for every $\eps>0$ there exists $0<\delta<1$ such that $|f(z)-f(0)|\le \eps$ for  all $f\in\mathcal F$ and all  complex  $z$ with $|z|\le \delta$.
\end{enumerate}

Then $\lim_{\a\to\infty} \int f\,d\nu_{\a}=f(0)$, uniformly for $f\in\mathcal F$.

In particular, $\lim_{\a\to\infty}\nu_\a=\delta_0$ weakly. 
\end{prop}

\begin{proof}
For any $0<\delta<1$ and $f\in\mathcal F$ we have
$$
\int f\,d\nu_\a-f(0)=\frac{\a+1}{\pi}\int_0^{2\pi}\left(\int_0^\delta +\int_\delta^1 (f(re^{i\t})-f(0)) r(1-r^2)^\a \,dr\right) d\t,
$$
hence
\begin{eqnarray*}
\left|\int f\,d\nu_\a-f(0)\right|&\le& \sup_{|z|\le\delta}|f(z)-f(0)|+ 2||f||_\infty\frac{\a+1}{\pi}\int_0^{2\pi}\int_\delta^1 r(1-r^2)^\a\,dr\,d\t\\
&=&  \sup_{|z|\le\delta}|f(z)-f(0)|+ 2||f||_\infty(1-\delta^2)^{\a+1}.
\end{eqnarray*}

For given $\eps>0$, we first choose $\delta>0$ so small that the first term is smaller than $\eps/2$.

With this $\delta$, the second term tends to zero as $\a\to\infty$, hence $\le\eps/2$ for $\a$ sufficiently large. 
\end{proof}

\medskip
\noindent {\bf Proof of Theorem~\ref{thm:main2}:}

If $f\in\mathcal P(\Omega_\infty,L)$, then for every $q\ge 2$
\begin{equation}
f(z,u) = \sum_{m,n=0}^\infty \varphi_{m,n}^{q-2}(u)R_{m,n}^{q-2}(z), \quad (z,u)\in\overline{\D}\times L,
\end{equation}
where $\varphi_{m,n}^{q-2}\in \mathcal P(L)$ satisfy
$$
\sum_{m,n=0}^\infty \varphi_{m,n}^{q-2}(e_L)<\infty,
$$
and
\begin{equation}
\varphi_{m,n}^{q-2}(u) = N(q;m,n)\int_{\overline{\D}}f(z,u)\overline{R_{m,n}^{q-2}(z)}d\nu_{q-2}(z)
\end{equation} 
by Theorem~\ref{thm:MP} and \eqref{eq:nual}.

There is a formula of  Rodrigues type for the disc polynomials, see \cite[Eq. (2.6)]{Wu}:
\begin{equation}\label{eq:rodrigues}
(1-|z|^2)^{q-2} R_{m,n}^{q-2}(z) = \frac{(-1)^{m+n}(q-2)!}{(m+n+q-2)!} \frac{\partial^{m+n}}{\partial \overline{z}^m \partial z^n}(1-|z|^2)^{m+n+q-2}.
\end{equation}
Thus, using $\overline{R_{m,n}^{q-2}(z)}=R_{n,m}^{q-2}(z)$,

\begin{eqnarray*}
\lefteqn{\varphi_{m,n}^{q-2}(u)=}\\ 
&& \frac{q-1}{\pi}\frac{(-1)^{m+n}(q-2)!}{(m+n+q-2)!}N(q;m,n)\int_{\overline{\D}}  f(z,u) \frac{\partial^{m+n}}{\partial \overline{z}^n \partial z^m}(1-|z|^2)^{m+n+q-2} dxdy.
\end{eqnarray*} 
Denote by $I$ the integral in the previous equation. Now, note that
\begin{eqnarray*}
I &=& \int_{\overline{\D}}  f(z,u) \frac{\partial}{\partial z}\left[\frac{\partial^{m+n-1}}{\partial \overline{z}^n \partial z^{m-1}}(1-|z|^2)^{m+n+q-2}\right] dxdy \\
 &=& \int_{\overline{\D}}   \frac{\partial}{\partial z}\left[f(z,u)\frac{\partial^{m+n-1}}{\partial \overline{z}^n \partial z^{m-1}}(1-|z|^2)^{m+n+q-2}\right] dxdy  \\
 &-& \int_{\overline{\D}}   \frac{\partial}{\partial z}f(z,u) \frac{\partial^{m+n-1}}{\partial \overline{z}^n\partial z^{m-1}}(1-|z|^2)^{m+n+q-2} dxdy.
\end{eqnarray*} 
By Green's Theorem,
\begin{equation}\label{eq:rel_z}
\int_{\overline{\D}} \frac{\partial g}{\partial \overline{z}}(z) dxdy = -\frac{i}2\int_{\partial\D} g(z) dz 
\end{equation}
and 
\begin{equation}\label{eq:rel_conj_z}
\int_{\overline{\D}} \frac{\partial g}{\partial {z}}(z) dxdy =\frac{i}2\int_{\partial\D} g(z) d\overline{z} 
\end{equation}
for a continuously differentiable function $g$ on $\overline{\D}$.
Using \eqref{eq:rel_conj_z} we get
\begin{eqnarray*}
\begin{split}
I &= \frac{i}2\int_{\partial\overline{\D}}  f(z,u)\frac{\partial^{m+n-1}}{\partial \overline{z}^n \partial z^{m-1}}(1-|z|^2)^{m+n+q-2} d\overline{z}\\   
 &- \int_{\overline{\D}}   \frac{\partial}{\partial z}f(z,u) \frac{\partial^{m+n-1}}{\partial \overline{z}^n \partial z^{m-1}}(1-|z|^2)^{m+n+q-2} dxdy.
\end{split}
\end{eqnarray*}
Since 
$$
\frac{\partial^{m+n-1}}{\partial \overline{z}^n \partial z^{m-1}}(1-|z|^2)^{m+n+q-2}
$$ 
is the product of a polynomial in $z$ and $\overline{z}$ by $(1-|z|^2)^{q-1}$, we have
$$
f(z,u)\frac{\partial^{m+n-1}}{\partial \overline{z}^n \partial z^{m-1}}(1-|z|^2)^{m+n+q-2} = 0, \quad z\in\partial\overline{\D}. 
$$
Therefore
$$
I = 
 - \int_{\overline{\D}}   \frac{\partial}{\partial z}f(z,u) \frac{\partial^{m+n-1}}{\partial \overline{z}^n \partial z^{m-1}}(1-|z|^2)^{m+n+q-2} dxdy,
$$
and similarly by \eqref{eq:rel_z}
$$
I = 
 - \int_{\overline{\D}}   \frac{\partial}{\partial \overline{z}}f(z,u) \frac{\partial^{m+n-1}}{\partial \overline{z}^{n-1} \partial z^{m}}(1-|z|^2)^{m+n+q-2} dxdy.
$$
We now make further integrations by parts, a total of $m$ integrations with respect to $z$ and $n$ with respect to $\overline{z}$. We need the following terms to vanish on the boundary of $\D$
\begin{equation}\label{eq:boundary1}
\frac{\partial^{l}f(z,u)}{\partial z^l}\frac{\partial^{m+n-l-1}}{\partial\overline{z}^n\partial z^{m-l-1}}(1-|z|^2)^{m+n+q-2}, \quad l=1,2,\ldots,m-1,
\end{equation} 
and
\begin{equation}\label{eq:boundary2}
\frac{\partial^{k+m}f(z,u)}{\partial \overline{z}^k\partial z^m}\frac{\partial^{n-k-1}}{\partial \overline{z}^{n-k-1}}(1-|z|^2)^{m+n+q-2}, \quad k=0,1,\ldots,n-1.
 \end{equation}
This is true because
$$
\frac{\partial^{r+s}}{\partial\overline{z}^r\partial z^s}(1-|z|^2)^{N}=p(z,\overline{z})(1-|z|^2)^{N-r-s},\quad r+s\le N
$$
for a polynomial $p$ in $z,\overline{z}$. Therefore the terms in \eqref{eq:boundary1},\eqref{eq:boundary2} are of the form $(1-|z|^2)^{q-1}F(z,u)$, where $F$ is continuous on $\overline{\D}\times L$ by Proposition~\ref{thm:Z2cp}. The terms then vanish on the boundary of $\D$ because $q\ge 2$.
 
We obtain
\begin{eqnarray*}
\lefteqn{\varphi_{m,n}^{q-2}(u)=}\\ 
&& \frac{q-1}{\pi}\frac{(q-2)!}{(m+n+q-2)!}N(q;m,n)\int_{\overline{\D}}  \frac{\partial^{m+n}}{\partial \overline{z}^n \partial z^m}f(z,u) (1-|z|^2)^{m+n+q-2} dxdy,\\
&=& N(q;m,n)\frac{(q-2)!}{(m+n+q-2)!}\int_{\overline{\D}} \frac{\partial^{m+n}}{\partial \overline{z}^n \partial z^m} f(z,u) (1-|z|^2)^{m+n} d\nu_{q-2}(z).
\end{eqnarray*} 
We have
\begin{eqnarray*}
\lefteqn{N(q;m,n)\frac{(q-2)!}{(m+n+q-2)!}=}\\
 &&\frac{1}{m!n!}\frac{q-1+m+n}{q-1}\, \frac{(q-2+m)!}{(q-2)!} \,\frac{(q-2+n)!}{(q-2+m+n)!}.
\end{eqnarray*}
Using 
$$
(a+n)! =a! (a+1)_n,
$$
  we find for  $a=q-2$
$$
N(q;m,n) \frac{(q-2)!}{(m+n+q-2)!}=
 \frac1{m!n!}\, \frac{q-1+m+n}{q-1} \, \frac{(q-1)_m (q-1)_n}{(q-1)_{m+n}},
$$
and then
$$
N(q;m,n) \frac{(q-2)!}{(m+n+q-2)!}\longrightarrow \frac1{m!n!}, \quad q\to\infty.
$$

The function
$$
h(z,u):= (1-|z|^2)^{m+n} \frac{\partial^{m+n}}{\partial \overline{z}^n \partial z^m} f(z,u)
$$
is continuous on $\overline{\D}\times L$ by Proposition~\ref{thm:Z2cp}, 
and therefore the  family $\mathcal F$ of the functions $h(\cdot,u)\in C(\overline{\D})$, 
where $u$ belongs to a compact subset of $L$, is bounded and equicontinuous at $z=0$. 

By Proposition~\ref{thm:Poincp} it follows that

$$
\varphi_{m,n}^{q-2}(u) \to \frac1{m!n!}\frac{\partial^{m+n}}{\partial \overline{z}^n \partial z^m} f(0,u), \quad q\to\infty,
$$
uniformly for $u$ in compact subsets of $L$.
\hfill$\square$

\begin{rem}\label{thm:remark} {\rm
It is known and easy to see that the disc polynomials $R^{\a}_{m,n}(z)$ have the following limit property
\begin{equation}\label{eq:disclim}
\lim_{\a\to\infty} R^{\a}_{m,n}(z)=z^m\overline{z}^n,\quad z\in\D
\end{equation}
for each $m,n\ge 0$ fixed, cf. \cite[(2.12)]{Wu}. 

This is the analogue of the following limit result for the normalized Gegenbauer polynomials
$$
\lim_{\la\to\infty}C_n^{(\la)}(x)/C_n^{(\la)}(1) =x^n,\quad -1<x<1
$$
for each $n\ge 0$. Schoenberg  \cite[p. 103]{S} proved that this convergence is uniform in $n\ge 0$ for fixed $x$, and this was the clue to his proof of the representation theorem for $\mathcal P(\S^\infty)$, cf. \cite[Theorem 2]{S}.

A proof of the theorem of Christensen and Ressel or the more general Theorem~\ref{thm:Schinftycp} can be given following the ideas of Schoenberg provided that one can prove that the convergence in \eqref{eq:disclim} is uniform in $m,n\ge 0$ for each fixed $z\in\D$.

We have not been able to settle this question, which is equivalent to the following property of the normalized Jacobi polynomials cf. \eqref{eq:Jacobinorm}
$$
\lim_{\a\to\infty} ((1+x)/2)^{\b/2}R_n^{(\a,\b)}(x)= ((1+x)/2)^{n+\b/2},\quad -1<x<1
$$ 
uniformly in $n,\b\in\N_0$.
}
\end{rem}

\noindent{\bf Acknowledgments}

The first two named authors want to thank  the Department of Mathematics at Universidad T{\'e}cnica Federico Santa Maria for hospitality during their visit.\\ 
The travel of the first author to Chile was supported by A. Collstrop's Foundation.
\\
The second author was partially supported by S{\~a}o Paulo Research Foundation (FAPESP) under grants 2016/03015-7 and 2014/25796-5.

\noindent
Christian Berg\\
Department of Mathematical Sciences, University of Copenhagen\\
Universitetsparken 5, DK-2100, Denmark\\
e-mail: {\tt{berg@math.ku.dk}}

\vspace{0.4cm}
\noindent
Ana P. Peron\\
Departamento de Matem{\'a}tica, ICMC-USP-S{\~a}o  Carlos\\
Caixa Postal 668, 13560-970 S{\~a}o Carlos SP, Brazil\\
e-mail: {\tt{apperon@icmc.usp.br}} 

\vspace{0.4cm}
\noindent
Emilio Porcu \\
Department of Mathematics, Universidad T{\'e}cnica Federico Santa Maria\\
Avenida Espa{\~n}a 1680, Valpara{\'\i}so, 2390123, Chile\\
e-mail: {\tt{emilio.porcu@usm.cl}}

\end{document}